\documentclass[11pt]{article}
\usepackage{amsmath,amsthm,amssymb,mathrsfs,amscd,amsthm,amscd,amsfonts}   					 	
\usepackage[cmtip,all]{xy} 	
\usepackage{graphicx,texdraw,epsfig,epic,eepic}  		
\usepackage[colorlinks,plainpages]{hyperref}      	
\usepackage{mathtools}

\makeatletter 
\def\blfootnote{\xdef\@thefnmark{}\@footnotetext} 
\makeatother

\newtheorem{theorem}{Theorem}[section]
\newtheorem{lemma}[theorem]{Lemma}
\newtheorem{proposition}[theorem]{Proposition }
\newtheorem{corollary}[theorem]{Corollary}

\theoremstyle{definition}
\newtheorem{definition}[theorem]{Definition}
\newtheorem{example}[theorem]{Example}

\theoremstyle{remark}
\newtheorem{remark}[theorem]{Remark}

\numberwithin{equation}{section}

\newcommand{\R}{{\mathbb R}}
\newcommand{\C}{{\mathbb C}}

\renewcommand{\P}{{\mathbb P}}

\newcommand{\G}{{\mathbb G}}
\newcommand{\clg}{{\mathbb G}}

\newcommand{\kc}{{\mathcal C}}
\newcommand{\kd}{{\mathcal D}}

\newcommand{\kh}{{\mathcal H}}

\newcommand{\km}{{\mathcal M}}

\newcommand{\ko}{{\mathcal O}}

\newcommand{\kt}{\mathcal{T}}

\newcommand{\kz}{{\mathcal Z}}
\newcommand{\fm}{\mathfrak{m}}

\newcommand{\sC}{\mathscr{C}}

\newcommand{\sK}{\mathscr{K}}

\newcommand{\sT}{\mathscr{T}}
\newcommand{\sU}{\mathscr{U}}
\newcommand{\sV}{\mathscr{V}}
\newcommand{\sX}{\mathscr{X}}

\newcommand{\sZ}{\mathscr{Z}}

\newcommand{\uDef}{\underline{\Def}\!\,}

\newcommand{\bm}{\boldsymbol{m}}
\newcommand{\bo}{\boldsymbol{0}}

\newcommand{\Iea}{I^{\text{\it ea}}}

\newcommand{\Ies}{I^{\text{\it es}}}

\newcommand{\eps}{{\varepsilon}}
\newcommand{\tea}{\tau^{\text{\it ea}}}
\newcommand{\tes}{\tau^{\text{\it es}}}
\newcommand{\Hilb}{\kh ilb}
\newcommand{\Dou}{\kd ou}

\DeclareMathOperator{\Cl}{\kc\!\!\;\ell\!\:}
\DeclareMathOperator{\mt}{mt}
\DeclareMathOperator{\pr}{pr}
\DeclareMathOperator{\tHilb}{Hilb\!\:}
\DeclareMathOperator{\tDou}{Dou\!\:}
\DeclareMathOperator{\Sym}{Sym}
\DeclareMathOperator{\Def}{\kd\!\!\:\it{ef}}
\begin{document}

\title{Straight Equisingular Deformations and Punctual Hilbert Schemes}
\author{Gert-Martin Greuel \footnote {Department of Mathematics, University of Kaiserslautern} \\[1.0ex]{\em Dedicated to L\^{e} D\~{u}ng Tr\'{a}ng on the occasion of his 70th Birthday}}



\date{}
\maketitle

\begin{abstract}
We study "straight equisingular deformations''\!, a linear subfunctor of all equisingular deformations  and describe their seminuniversal deformation by an ideal containing 
the fixed Tjurina ideal.  
Moreover, we show that the base space of the seminuniversal straight equisingular deformation
appears as the fibre of a morphism from the $\mu$-constant stratum onto a punctual Hilbert scheme parametrizing certain zero-dimensional schemes concentrated in the singular point.
Although equisingular deformations of plane curve singularities are very well understood, we believe that this aspect may give a new insight in their inner structure. 
\end{abstract}

\addcontentsline{toc}{section}{Introduction}

\vspace{1.0cm}

\begin{center}
\textbf{\large Introduction}
\end{center}

Let $z$ be a singular point of the reduced curve $C$ contained in the smooth complex surface $\Sigma$. The topological type of the germ $(C,z)$ can be described by the resolution graph of a good embedded resolution of  $(C,z)$, and an equisingular deformation of $(C,z)$  is one with constant resolution graph. In the first section we fix the notations and recall the classical concepts of constellation, clusters and proximities, resulting in the notions of the essential tree 
$\sT^*(C,z)$ and the cluster graph $\G(C,z)$, which is  also a complete invariant 
of the toplogical type of $(C,z)$. 
These notions go back to Enriques and Chisini, for a comprehensive and modern treatment of we refer to \cite {Cas}.

A deformation of  $(C,z)$ is equisingular or an es-deformation, if 
there exist equimultiple sections through the points of the essential tree (i.e. through the non-nodes of a minimal good resolution) of $(C,z)$ such that the family can be blown 
up successively along these sections. Those deformations for which these sections can be simultaneously trivialized are called straight equisingular or straight es. These deformations have been considered by J. Wahl in \cite{Wahl} (not under this name) and they were further characterized in \cite{GLS07}, showing, e.g., that for semiquasi-homogeneous and Newton nondegenerate singularities every equisingular deformation is straight equisingular. 
In Section 2 we introduce, in addition, for any plane curve singularity 
the topological singularity ideal $I^s (C,z)$ such that $I^s (C,z) / \Iea_{fix}(C,z)$, 
with $\Iea_{fix} (C,z)$ the fixed Tjurina ideal, is the tangent space to straight equisingular deformations. This provides an explicit description of the semiuniversal straight equisingular deformation as a linear subspace of all equisingular deformations. We illustrate this by an example using {\sc Singular} \cite{DGPS}.

In Section 3 we consider the punctual Hilbert scheme, parametrizing zero-dimensional schemes $Z \subset \Sigma$ with support in one point.  We are interested in the {\em rooted Hilbert scheme}, a subscheme of the punctual Hilbert scheme consisting 
of the {\em topological singularity scheme} $Z^s(C,z)$ defined by the  ideal $I^s (C,z)$, with fixed cluster graph $\G = \G(C,z)$. Mapping an equisingular family of plane curve singularities to the induced family of topological singularity schemes, provides a morphism from 
the functor of all es-deformations to the rooted Hilbert functor. We prove that the fibre of the induced morphism from the base space of the semiuniversal es-deformation of $(C,z)$ to the rooted Hilbert scheme  is exactly the base space of the semiuniversal straight equisingular deformations.
This gives in addition a nice formula for the rooted Hilbert scheme in terms of the number of free vertices in the essential tree $\sT^*$ of $(C,z)$.

Let us finish the introduction by commenting on the positive characteristic case.
Though we work here over the complex numbers, the results should hold also for algebraically closed ground fields of characteristic zero.  Equisingular deformations of plane curve singularities in arbitrary characteristic have been studied  in \cite{CGL1} and  \cite{CGL2}. In positive characteristic one has to distinguish between ``good''  and ``bad'' characteristic and we conjecture that the results of this paper can be extended to good characteristic as well.  However, in bad characteristic one has to distinguish between equisingular deformations and weakly equisingular deformations, i.e., those that become equisingular after a finite base change, with very different semiuniversal deformations. 
It would be interesting to study straight es-deformations in positive characteristic, in particular in bad characteristic.
\medskip

{\bf Acknowledgement:}  The article is a revised excerpt of parts from Chapter I.1 of the forthcoming book \cite{GLS18} by Chistoph Lossen, Eugenii Shustin and the author, where more details can be found.
 I like to thank C. Lossen and E. Shustin for the many pleasant years of joint work on the book. 
\medskip

\section{Equisingularity for Plane Curve Singularities}
Before we define in the next section straight equisingular deformations, recall the
definition of an equisingular deformation of a reduced plane complex curve singularity $(C,z) \subset (\Sigma,z)$. 
\medskip

It is well-known that
we can resolve the singularity  $(C,z)$ by finitely many
blowing ups of points. More precisely, there exists a {\em good
embedded
resolution}\index{resolution!good embedded}
\index{embedded!resolution}\index{good!embedded resolution} of the singularity of $C$ at $z$, that is, a sequence of morphisms of smooth surfaces
\[
\pi: \Sigma_{n+1}\xrightarrow{\pi_{n+1}}\Sigma_n\xrightarrow{\pi_n}\cdots \to
\Sigma_1\xrightarrow{\pi_1}\Sigma_0=\Sigma,
\]
such that $\pi_i$ is the blowing up of a point \mbox{$q_{i-1}\in\Sigma_{i-1}$}
infinitely near to
\mbox{$q_0=z$} and such that in a neighborhood of
\mbox{$E=\pi^{-1}(z)$}, the {\em reduced total transform} of $C$
\[
\pi^\ast(C)_{red}=\widetilde{C}+\sum\limits_{i=1}^{n+1} E_i
\]
is a divisor with
normal crossings\index{divisor!with normal crossings}, that is, a
hypersurface having only nodes as singularities (see \cite[Section
I.3.3]{GLS07} for details).
Here $\widetilde{C}$ is the (smooth) {\em strict transform} of $C$,  $\pi^\ast(C)$ (with its scheme structure) is the {\em total transform} and $E_i = \pi_i^{-1}(q_{i-1}) \cong \P^1$, $i=1, \ldots, n+1$, is the {\em exceptional
divisor} of $\pi_i$ in $\Sigma_i$ (identified with its image in $\Sigma_j$, $j\ge i$). 
A good embedded resolution $\pi$ is called a {\em minimal (good) embedded resolution}, 
if only non-nodal singular points of the reduced total
transform of $(C,z)$ are blown up in the resolution process (if  $(C,z)$ is smooth, we do not blow up anything, i.e. $\pi$ is the empty map). It is well-known
that such a minimal good resolution is unique up to isomorphism over $\Sigma$.

For $q\in\pi^{-1}(z)$ we denote by 
$C_{(q)}, \widehat C_{(q)}$ and $\widetilde C_{(q)}$ the respective
germs at $q$ of the strict, the total  and the reduced total transform of $C$  and and by $\mt C_{(q)}, \mt\widehat C_{(q)}$ and $\mt\widetilde C_{(q)}$ their multiplicities. 
If the germ $C_{(q)}$ is non-empty,
we say that the curve $C$ {\em goes through} the infinitely near
point $q$ (or that $q$ {\em belongs to} $C$).
\medskip

Recall that classically $E_1 \subset \Sigma_1$ and its strict transforms in $\Sigma_i, \ i>1$, together with $z$, is called the {\em first infinitesimal neighbourhood}
\index{first infinitesimal neighbourhood} \index{infinitesimal neighbourhood!first}
of $z\in \Sigma$. For $i > 1$ the $i$-th infinitesimal neighbourhood of $z$  consists of points in the first infinitesimal neighbourhood of some point in the
$i-1$-st infinitesimal neighbourhood of $z$. Any point belonging to some infinitesimal neighbourhood of $q_i, i\geq 0,$ is called an {\em infinitely near point} 
of $q_i \in \Sigma_i$ or a point {\em infinitely near to\/} $q_i$.

\begin{definition}\label{def:constellation}{\bf (Constellation)}
If  $(q'_0,\pi'_1,q'_1,\dots,\pi'_m,q'_m)$  with $q'_i\in \Sigma'_i$ and
 $\pi'_i:\Sigma'_i\to \Sigma'_{i-1}$ blowing up $q'_{i-1}$
is another sequence of  a good embedded resolution of $(C,z)$, then
$(q_0,\pi_1,q_1,\dots,\pi_n,q_n)$
and   $(q'_0,\pi'_1,q'_1,\dots,\pi'_m,q'_m)$
are called equivalent if $m=n=-1$ or $m=n\geq 0$ and
if there is a $\Sigma$-isomorphism $\Sigma_{n+1} \to \Sigma'_{n+1}$.
An equivalence class of such sequences is called a {\em constellation} (of $(C,z)$)
and denoted by $ \sT(C,z)$. 
\end{definition}
It is easy to see that the $\Sigma$-isomorphism preserves infinitely nearness and 
that successively blowing up two different points in the two possible orders gives rise to $\Sigma$-isomorphic surfaces,  and hence the same constellation.
The concept of a constellation has been introduced to
describe isomorphism classes of good embedded resolutions.

\begin{definition}\label{def:tree}{\bf (Constellation graph, essential tree, proximity)}
Let $\sT =  \sT(C,z) = (q_0,\pi_1,q_1,\dots,\pi_n,q_n)$ be a constellation of $(C,z)$. 
\medskip

(1) We introduce a partial ordering on the points $q_0,\dots,q_n$ by
$$ q_i \leq q_j \;\: :\Longleftrightarrow\;\: q_j \text{ is infinitely near to } q_i\,.$$
For each $0 \leq i \leq n$, we set the
{\em level of} $q_i$ as $ \ell(q_i) := \# \bigl\{ j \!\;\big|\;\! q_i \geq q_j\bigr\} -1\,.$ 
The point $z=q_0$, the {\em origin of} $\sT$, is the only point of level $0$. 
\medskip

(2) The   {\em graph of the constellation $\sT$} 
is the oriented tree  $\Gamma_\sT$ (with root $z=q_0$) whose
\begin{itemize}
\item points are in 1-1 correspondence with $q_0,\dots,q_n$ and
 \item edges with pairs $(q_j,q_i)$ s.t. $\ell(q_j)=\ell(q_i)+1$ and $q_j\geq q_i$.  
\end{itemize}
Hence $p \geq q$ iff there is an oriented path in 
$\Gamma_\sT$ from $p$ to $q$.
\medskip

(3) A point $z\neq q\in \sT(C,z)$ is called {\em  essential}
for $C$ iff the reduced total transform of $C$ does not have a node at $q$. The origin $z$ of $\sT(C,z)$ is called essential for $C$, 
iff the germ $(C,z)$ is not smooth.  We call a point $q\in \sT(C,z)$ a
{\em singular essential point\/}\index{essential!point!singular} for $C$ if
the strict transform of $C$ at $q$ is singular.
The maximal finite sub-constellation of $\sT(C,z)$ such that all of its
points are essential is called the {\em essential constellation} of $(C,z)$ and denoted by $ \sT^\ast (C,z)$. 

By abuse of notation, a constellation  $\sT$ is also called a tree and $ \sT^\ast (C,z)$ is called the {\em essential tree} of $(C,z)$.
It describes the minimal embedded resolution of $(C,z)$.
\medskip

(4) Finally, we call 
$q_j$ {\em proximate to} $q_i$ (notation $q_j\dashrightarrow q_i$), if $q_j$ is a point 
in $E_{q_i}$, the exceptional divisor of $\pi_{i+1}$, or in any of its strict transforms.
$q_j$ is a {\em satellite} point if it is proximate to (at least) two
 points $q_i$, \mbox{$0\leq i \leq j\!\!\:-\!\!\:1$},
otherwise it is {\em free}. 
\end{definition}

It is clear that a point $q_j$ cannot be proximate to more than two
points since the exceptional divisors have normal crossings. Notice that each satellite point in
$\sT(C,z)$ is an essential point. It may well be a non-singular essential
point for $C$.

\medskip

It is classically known that the topological type of $(C,z)$ can be characterized by several different sets of discrete data (see \cite[Section I.3.4]{GLS07}
 for a short overview). One characterization is by the graph 
$\Gamma_{\sT}$,  $\sT = \sT(C,z)$ a constellation of $(C,z)$, 
together with the multiplicities 
$\mt\widehat C_{(q)}, \ q \in \sT$, of the total transform (or, equivalently, the multiplicities 
$\mt\widetilde C_{(q)}$ of the reduced total transform). 
This weighted oriented graph is sometimes called a 
{\em resolution graph} of $(C,z)$.

\begin{remark}\label{rem:proximityeq}
We have by \cite[Proposition I.1.11]{GLS18} the following {\em proximity equality}:
\begin{equation*}
\label{proximity equality}
\mt C_{(p)}\: =  \!\!\!\sum_{
\renewcommand{\arraystretch}{0.5}
\begin{array}{c}
\scriptstyle{ q \in \sT} \\
\scriptstyle{q\dashrightarrow p }
\end{array}
}\!\!\!
\mt C_{(q)}\,.
\end{equation*}
The difference between the multiplicities of
the total and strict transform of $C$ at a point $q\in \sT$ is
$$  \mt \widehat{C}_{(q)} -\mt C_{(q)} \:=
\!\!\sum_{\renewcommand{\arraystretch}{0.5}
\begin{array}{c}
\scriptstyle{p \in \sT} \\
\scriptstyle{q \dashrightarrow p }
\end{array}}\!\! \mt \widehat{C}_{(p)} $$
(see \cite[Remark I.1.11.1]{GLS18}). Moreover 
$ \mt \widetilde{C}_{(q)} -\mt C_{(q)} = $ 1 or 2, depending if one or two 
points are proximate to $q$.
In particular, the multiplicities of the strict transforms of $C$
 together with the proximities (\mbox{$q\dashrightarrow p$}) determine
the multiplicities of the total transforms and hence can be used to describe the topological type of $C$. This is used in the following definition.
\end{remark}

\begin{definition}\label{def:cluster}{\bf (Cluster, Cluster graph)}
Let  $\sT^\ast = \sT^\ast(C,z) =(z\!\!\:=\!\!\:q_0,\pi_1,q_1,\dots,\pi_n,q_n)$ be the essential tree of $(C,z)$ and 
${\bm}=(m_0,\dots, m_n)$ the vector of multiplicities with $m_0:=\mt(C,z)$ and $m_i:=\mt C_{(q_i)}$, the
multiplicity of the strict transform of $C$ at $q_i$, $1\leq i\leq n$. We call the tuple
$(\sT^\ast,{\bm})$ the {\em cluster} of $(C,z)$.
The triple
$$\G(C,z) = \G(C,\sT^\ast) := (\Gamma_{\sT^\ast},\dashrightarrow, {\bm}),$$
consisting of $\Gamma_{\sT^\ast}$,  the graph of $\sT^\ast$, 
the binary relation defined by $q\dashrightarrow p$
if $q$ is proximate to $p$ and the vector $\bm$ of multiplicities of the strict transforms
is called  the {\em cluster graph} of $(C,z)$. 
\end{definition}
Notice that the cluster graph
$\G(C,z)$ determines (and is determined by) the topological type  the curve singularity $(C,z)$.
\\


We consider now an {\em embedded
deformation}\index{deformation!embedded} \index{embedded deformation}
$(i,\Phi,\sigma)$ of
$(C,z)\subset(\Sigma,z)$, {\em with
section}\index{deformation!embedded!with section} over a complex
germ $(T,0)$. This is a commutative diagram
$$
\xymatrix@C=16pt@R=24pt{ \ (C,z)\ \ar@{^{(}->}^-{i}[r]\ar[d] &
({\mathcal C}, z)\ar[d]^-{\Phi}\ \ar@{^{(}->}[r]
& ({\mathcal M},z)\ar[dl]^-{
} \\
\ \{0\}\ \ar@{^{(}->}[r] & (T,0)\ar@/^/[u]^-{\sigma}}
$$ of morphisms of complex space germs with $({\mathcal C},z)\subset({\mathcal
M}, z)$ a hypersurface germ, $\sigma$ a section of $\Phi$. $\Phi$
is assumed to be flat as well as $({\mathcal M},z)\to(T,0)$ which has $(\Sigma,z)$ a
special fibre. An embedded deformation (without section), denoted by
$(i,\Phi)$, is given by a diagram as above but with $\sigma$
deleted. We usually choose small representatives for the germs and
denote them with the same letters, omitting the base points.
Note that, for small representatives, we have ${\mathcal M}\simeq\Sigma\times T$ over $T$,
and we have morphism $C\overset{i}{\hookrightarrow}{\mathcal
C}\overset{\Phi}{\to}T$ with fibres the reduced curves ${\mathcal C}_t=\Phi^{-1}(t)$, where
we identify $C$ with ${\mathcal C}_0$ and write $z$ instead of $(z,0)\in \Sigma \times T$.

 If ${\mathcal M}=\Sigma\times T$ with $\Phi$ the projection and if
$\sigma(t)= (z,t)$, then $\sigma$ is called the {\em trivial
section}. By \cite[Proposition II.2.2]{GLS07} every section can be
locally trivialized by an isomorphism ${\mathcal M}\simeq\Sigma\times T$
over $T$.
\medskip

The following definition of equisingularity makes sense for arbitrary complex germs $(T,0)$, even for Artinian ones. Let $(C,z)$ be defined by
$f\in{\mathcal O}_{\Sigma,z}$ with multiplicity $m=\mt(f)$. 

\begin{definition}\label{def:es-deformation}{\bf (Equisingular deformation)}  
If $(C,z)$ is smooth, any deformation of $(C,z)$ is called equisingular. If $(C,z)$ 
is singular let ${\sT}^*={\sT}^*(C,z)$ be the essential tree of $(C,z)$ and
for each $q\in{\sT}^*, \, q \neq z,$ let $\widetilde C_{(q)}$ be the germ at $q$ of the reduced total transform of $(C,z)$.
An embedded deformation $(i,\Phi)$ of $(C,z)$ over $(T,0)$ is then called
{\em equisingular}\index{equisingular!deformation}\index{deformation!equisingular}, or an {\em es--deformation}\index{es-deformation}, if the following three conditions hold:
\begin{enumerate}\item[(1)] There exists a section
$\sigma$ of $\Phi$, called {\em equimultiple section},
\index{section!equimultiple}\index{equimultiple!section}such that
$\Phi$ is equimultiple along $\sigma$.  If $({\mathcal
C},z)$ is defined by $F\in{\mathcal O}_{{\mathcal M},z}$ this means that $F\in
I^m_\sigma$, where $I_\sigma$ denotes the ideal of
$\sigma(T,0)\subset({\mathcal M},z)$.

\item[(2)] For each $q\in{\sT}^*, \, q \neq z,$ there is a sequence of morphisms of germs (or of small representatives)
$$\Phi_{(q)}:\widetilde{\mathcal C}_{(q)}\hookrightarrow{\mathcal
M}_{(q)}\overset{\pi_{(q)}}{\to}{\mathcal M}\to T,$$ 
where $\Phi_{(q)}: \widetilde{\kc}_{(q)} \to T$
 is an embedded deformation of $\widetilde C_{(q)}\subset\Sigma_{(q)}$ 
along an equimultiple section
$\sigma_{(q)}:T\to\widetilde{\mathcal C}_{(q)}$ of $\Phi_{(q)}$. 

\item[(3)]  Each ${\mathcal M}_{(q)}$, $q\ne z$, is obtained by
blowing up some ${\mathcal M}_{(p)}$ along the section $\sigma_{(p)}$, 
$p$ of smaller level than $q$, with $\widetilde{\mathcal C}_{(q)}$
the reduced total transform of $\widetilde{\kc}_{(p)}$ (for $p=z$ we take
$\km_{(p)} =\km, \, \sigma_{(p)}= \sigma$ and $\widetilde{\kc}_{(p)} = \kc$).
\end{enumerate}
\end{definition}

See \cite[Definition II.2.6]{GLS07} for a more detailed description.

\begin{remark} \label{rm:es-deformation}
(1) Condition (1) of Definition \ref{def:es-deformation} implies that for $t\in T$ sufficiently close to $0$ we have 
$F_t \in I^m_{\sigma(t)}$, where $F_t$ defines the germ $({\mathcal C}_t,\sigma(t))$ in $({\mathcal M}_t,\sigma(t))$ and we have
$I_{\sigma(t)}=\fm_{{\mathcal C}_t,\sigma(t)}$. Hence, the multiplicity of
$({\mathcal C}_t,\sigma(t))$ is constant for $t$ near $0$. For reduced base
spaces this is equivalent to the given definition of equimultiplicity.

(2) Recall that, for an equisingular deformation, the equimultiple sections 
$\sigma_{(q)}$ through all essential points are unique (cf.\cite[Proposition II.2.8]{GLS07}).
That is, after blowing up an equimultiple section, there is a unique section along which the blown up family is equimultiple.

(3) Let $(i,\Phi,\sigma)$ be an embedded deformation of
$(C,z)\subset(\Sigma,z)$ along the section $\sigma$ over a reduced 
complex germ $(T,0)$. Then $\Phi$ is equisingular iff the cluster graph $\G({\mathcal C}_t,\sigma(t))$ is constant for $t \in T$.
\end{remark}

\section{Straight Equisingular Deformations} \label{s1.2.4}

We consider now a special class of equisingular deformations, originally introduced by Wahl, which we call {\em straight}.

Assume that the (unique) equimultiple sections $\sigma_{(q)}:T\to\widetilde{\mathcal C}_{(q)}$ from
Definition \ref{def:es-deformation} for an es-deformation of $(C,z)$ over a complex germ $(T,0)$
are all trivial. We know that they can
be trivialized for each $q$ by a local isomorphism of germs ${\mathcal
M}_{(q)}\simeq\Sigma_{(q)}\times T$ over $T$ at $q$, but in general not
simultaneously for all $q$ by an isomorphism of $({\mathcal M},z)$ over $T$
(e.g. the cross-ratio of more than three sections through one exceptional component
is an invariant).
Those deformations for which these sections can be simultaneously trivialized are called straight equisingular.

\begin{definition}\label{straight}{\bf (Straight equisingular deformation)}
A deformation with section of the reduced plane curve singularity $(C,z)$ is called
{\em straight equisingular}\index{equisingular!deformation!straight}\index{straight equisingular} or a {\em straight es-deformation}
if it is an equisingular deformation of $(C,z)$ along the trivial section,
such that the equimultiple sections $\sigma_{(q)}, q \in \sT^*(C,z)$, through the non-nodes of the reduced total transform of $(C,z)$ from Definition \ref{def:es-deformation} are all trivial.

We denote by $\Def^s_{(C,z)}$ the category of straight es-deformations, the full
subcategory of the category $\Def^{sec}_{(C,z)}$ of deformations with section, and by $\uDef^s_{(C,z)}$ the functor of isomorphism  classes  of straight es-deformations.
\end{definition}

Let us give a concrete description of straight equisingular deformations, using the notations from Definition \ref{def:es-deformation}:

Denote by
$\widetilde F_{(q)}\in{\mathcal O}_{{\mathcal M}_{(q)}}$ a generator of the
ideal of $\widetilde{\mathcal C}_{(q)}\subset{\mathcal
M}_{(q)}=\Sigma_{(q)}\times T$ (the reduced total transform) and by $\fm_{(q)}\subset{\mathcal
O}_{\Sigma_{(q)}}$ the maximal ideal. The condition that
$\sigma_{(q)}$ is the trivial equimultiple section of $\Phi_{(q)}$ is
equivalent to $$\widetilde F_{(q)}\in\fm_{(q)}^{\widetilde
m(q)}\cdot{\mathcal O}_{\Sigma_{(q)}}\times T\ ,$$ where $\widetilde
m_{(q)}$ is the multiplicity of $\widetilde C_{(q)}$.

If $\widehat F_{(q)}\in{\mathcal O}_{{\mathcal M}_{(q)}}$ defines the total
transform $\widehat {\mathcal C}_{(q)}\subset\Sigma_{(q)}\times T$ and if
$\widehat m_{(q)}$ is the multiplicity of the total transform $\widehat
C_{(q)}$ of the curve germ $(C,z)$, then this is also equivalent to 
$$\widehat F_{(q)}\in\fm_{(q)}^{\widehat m(g)}\cdot{\mathcal O}_{\Sigma_{(q)}\times T}.$$ 
This can be seen easily by induction on the number of
blowing ups to resolve $(C,z)$, using 
Remark \ref{rem:proximityeq}. For $q=z$, we understand
$\widetilde{\mathcal C}_{(q)}=\widehat{\mathcal C}_{(q)}={\mathcal C}$, defined by
$F\in{\mathcal O}_{\Sigma\times T}$, and both conditions mean
$F\in\fm_z^m\cdot{\mathcal O}_{\Sigma\times T}$ with $m=\mt(C,z)$.
\smallskip

So far everything works for an arbitrary complex germ $(T,0)$. If
$(T,0)$ is reduced then the equimultiplicity condition for the trivial sections
$\sigma_{(q)}$, $q\in{\kt}^*$, is equivalent to
$$\mt(C,z)=\mt({\mathcal C}_t,z)\quad\text{for}\quad q=z$$
 and for $q\neq z$ either to
$$\mt\widetilde C_{(q)}=\mt(\widetilde{\mathcal C}_{(q),t},q)$$
or (equivalently) to
$$\mt\widehat C_{(q)}=\mt(\widehat{\mathcal C}_{(q),t},q)$$
for all $t\in T$ sufficiently close to
$0$. $\widetilde{\mathcal C}_{(q),t}$ resp. $\widehat{\mathcal
C}_{(q),t}$ denotes the reduced total,
resp. the  total transform of
the fibre ${\mathcal C}_t$ of $\Phi$ over $t$.

\bigskip

For $(T,0)=(\C,0)$ we can describe
the straight equisingularity condition even more explicitly. We do this for
the total transform $\widehat F_{(q)}\in{\mathcal O}_{\Sigma_{(q)}\times T}$
of $F\in{\mathcal O}_{\Sigma\times T}={\mathcal O}_{\Sigma,z}\{t\}$. Then
$F$ can be written as $$F=f+tg_1+t^2g_2+...$$ with $f,g_i\in{\mathcal
O}_{\Sigma,z}$ and $f$ defining $(C,z)$. Then $\widehat F_{(q)}$ reads as
$$\widehat F_{(q)}=\widehat f_{(q)}+t\widehat g_{1,(q)}+t^2\widehat g_{2,(q)}+...$$
where $\widehat f_{(q)},\widehat g_{i,(q)}\in{\mathcal O}_{\Sigma_{(q)}}$ denote
the total transforms of $f$, $g_i$.

If we fix $t=t_0\in T$ we write $F_{t_0}=F\big|_{t=t_0}\in \ko_{\Sigma,z}$ and $\widehat
F_{(q),t_0}=\widehat F_{(q)}\big|_{t=t_0}\in \ko_{\Sigma(q), q}$.
Then the equisingularity condition is equivalent to
$$\mt(f)=\mt(F_t)\ ,$$ $$\mt(\widehat f_{(q)})=\mt(\widehat F_{(q),t})$$ for
all $t\in T$ sufficiently close to $0$ and for all $q\in{\mathcal T}^*, q\neq z$.

Thus we get:

\begin{lemma}\label{ln1.23}
Let $f\in \C\{x,y\}$ define a reduced plane curve singularity $(C,0)$
with essential tree ${\sT}^\ast$ and let
$F(x,y,t)=f(x,y)+\sum_{i\geq 1} t^i g_i(x,y)$ define a
one-parametric deformation of $(C,0)$.
Then the following are equivalent:
\begin{enumerate}
\item [(i)]  The deformation of $(C,0)$ defined by $F$ is equisingular
and the (unique) equimultiple sections through the infinitely near
points $q\in{\sT}^*$ are trivial, i.e., the deformation is straight equisingular.

\item [(ii)] For all $i\geq 1$,
$$\mt(f)\le\mt(g_i),\ $$ $$\mt(\widehat
f_{(q)})\le\mt(\widehat g_{i,(q)})\ for \ q \in {\sT}^\ast, q\neq 0.$$

\end{enumerate}
\end{lemma}

\medskip
We define now an ideal defining the topological singularity scheme, that will be identified in Corollary \ref{cor:Is} with the tangent ideal to straight es-deformations.

\begin{definition}\label{def:topsing} {\bf (Topological singularity ideal and scheme)}
  Let $(C,z)\subset (\Sigma,z)$ be a reduced plane curve singularity and
  let $\sT^* = \sT^*(C,z)$ be the essential tree of $(C,z)$.
The ideal 
\begin{eqnarray*}
 I^s(C,z) & := &  I(C,\sT^*(C,z)) \subset \ko_{\Sigma,z}\\
& := & \bigl\{g \in \ko_{\Sigma,z} \:\big|\: \mt \widehat{g}_{(q)}
\geq \mt \widehat{C}_{(q)}\!\;\text{ for each } q \in \sT^*\bigr\}
\end{eqnarray*}
is called the {\it topological singularity ideal of} $(C,z)$. It defines 
the {\em topological singularity scheme} 
$$Z^s(C,z):=V(I^s(C,z)),$$
a zero-dimensional subscheme of $\Sigma$ supported
at $\{z\}$.
\end{definition}

\begin{remark} \label{rm1.30.1}
Let $\sT=\sT(C,z)$ be any constellation of $(C,z)$. Then we can define 
more generally the {\em cluster ideal} of $(C,z)$ w.r.t. $\sT$,
$$ I(C,\sT) = \bigl\{g \in \ko_{\Sigma,z} \:\big|\: \mt \widehat{g}_{(q)}
\geq \mt \widehat{C}_{(q)}\!\;\text{ for each } q \in \sT\bigr\},$$
and $Z(C,\sT) = V( I(C,\sT))$ the {\em cluster scheme} 
of $(C,z)$ w.r.t. $\sT$, which is 
supported at $\{z\}$.
\end{remark}

The following lemma can be proved by induction on the cardinality of $\sT^*(C,z)$,
see \cite[Lemma I.1.22]{GLS18}.

\begin{lemma}\label{ln1.25}
For $(C,z)\subset(\Sigma,z)$ a reduced plane curve singularity we have
$$\deg Z^s(C,z)=\dim_{\C}{\mathcal O}_{\Sigma,z}/I^s(C,z)=
\sum_{q\in\sT^{\ast}(C,z)}\frac{m_q(m_q+1)}{2}\ ,$$
with $m_q=\mt C_{(q)}$, the multiplicity of the strict transform of $C$ at $q$.
\end{lemma}

\begin{example}\label{ex1.24.1}
(a)  Let $(C,z)$ be smooth. For the empty constellation $\sT$ we obtain
$I(C,\sT) = \ko_{\Sigma,z}$, that is, the scheme
  \mbox{$Z(C,\sT)$} is the empty scheme. If
$(C,z)$ has the local equation $y=0$ and if 
$\sT=( z \!=\! q_0,q_1,\dots, q_n)$  is the constellation obtained by 
blowing up $(C,z)$ $n$ times, then
$I(C,\sT) = \langle y, x^{n+1}\rangle\subset \C\{x,y\}$.

(b) If $(C,z)$ is an {\em ordinary} $m$-fold {\em singularity} (i.e. $m$ smooth
  branches with different tangents) then $\sT^\ast(C,z) = (z)$ and
$I^s(C,z) =  \mathfrak{m}^m_{\Sigma,z}$.
\end{example}

\noindent
The following lemma shows the relation of cluster schemes to
equisingular deformations of curve germs:

\begin{lemma}\label{1.4}
Let $Z^s(C,z)$ be the topological singularity scheme of $(C,z)$ defined by 
$I^s(C,z)$.
\begin{enumerate}
\item[(a)] A generic element $g\in I^s(C,z)$ defines $Z^s(C,z)$, in the sense that
$Z^s(C,z)$ = $Z^s(C',z)$ for $(C',z)$ the plane curve singulariy defined by $g$.

\item[(b)] The elements $g\in  I^s(C,z)$ defining $Z^s(C,z)$ have no common
  (infinitely near) base point outside of $\sT^*(C,z)$.
\item[(c)] Two generic elements $g,g'\in I^s(C,z)$ are topologically
  equivalent.
\end{enumerate}
\end{lemma}

Note that ``generic'' in (a) means: there exists a polynomial defining $Z^s(C,z)$, 
and if $d_0$ is the minimal degree of such a polynomial, then, for each
$d\geq  d_0$, the set of polynomials $g\in I^s(C,z)$ of degree at most $d$ defining $Z^s(C,z)$ is a Zariski-open, dense subset in the vector space
of all polynomials of degree at most $d$ and contained in $I^s(C,z)$.

\begin{proof}
(a) Let  \mbox{$f\in  \ko_{\Sigma,z}$} be a defining equation for $(C,z)$. As $f$ is (analytically) finitely determined, the essential tree $\sT^*(C,z)$
depends only on a sufficiently high jet of $f$. We may therefore assume 
that $f$ is a polynomial of degree $d$. Then the polynomials $g\in I^s(C,z)$
of degree $\leq d$ are parameterized by a finite dimensional vector space
of positive dimension. The conditions $\mt \widehat{g}_{(q)} =
  \mt \widehat{C}\!\,'_{(q)}$, $q\in \sT^*(C',z)$, define a
Zariski-open subset of it. The density follows, since for almost all
$t\in \C$ the germ $f\!\!\;+\!\!\;tg$ has precisely the
 multiplicities $\widehat{m}_q = \mt  \widehat{C}\!\,'_{(q)}$ at each $q\in \sT(C',z)$.

For the proof of (b) we refer to Proposition  \cite[Proposition I.1.17]{GLS18}. (c) follows from (a), since $Z^s$ determines the topological type.
\end{proof}
\medskip

We relate below the topological singularity ideal $I^s(C,z)$  to the {\em equisingularity ideal}
$$\Ies(C,z)= \Ies(f):=\left\{g\in\ko_{\Sigma,z}\,\left|\,
\renewcommand{\arraystretch}{1}
  \begin{array}{c}
f+\eps g\ \text{is an {\it es }-deformation of}\ (C,z)\\
\text{over}\ T_\eps\ \text{(along some section)}
\end{array}\right.
\!\!\:\right\}$$
resp., the {\em fixed equisingularity ideal}
$$\Ies_{\text{fix}}(C,z)= \Ies_{\text{fix}}(f):=\left\{g\in\ko_{\Sigma,z}\,\left|\,
\renewcommand{\arraystretch}{1}
  \begin{array}{c}
f+\eps g\ \text{is an {\it es }-deformation of}\ (C,z)\\
\text{over}\ T_\eps\ \text{along the trivial section}
\end{array}\right.
\!\!\:\right\}.$$
They satisfy
$$\Ies(C,z) =  j(f) + \Ies_{fix}(C,z).$$ 
Moreover, if  $j(f)$ denote the Jacobian ideal we call 
 $$\Iea(C,z) := \langle f,j(f)\rangle \ \text{ resp. } \ 
\Iea_{fix}(C,z) :=  \langle f,\fm_zj(f)\rangle$$
the {\em Tjurina ideal} resp. the {\em fixed Tjurina ideal}. 
 $T_\eps$ denotes the fat point with structure sheaf $\C[\eps], \ \eps^2=0$, 
and  $\uDef^-_{(C,z)}(T_\eps)$
 is the tangent space to the deformation functor $\uDef^-_{(C,z)}$, i.e.,
 $$\uDef^{-}_{(C,z)}(T_\eps) = \{ g\in \ko_{\Sigma,z} | f+ \eps g 
\in \Def^{-}_{(C,z)}(T_\eps) \} / \text {isomorphism in } \Def^{-}_{(C,z)}.$$
\smallskip

\begin{corollary}\label{cor:Is}
Let \mbox{$(C,z)\subset (\Sigma,z)$} be a reduced plane curve
singularity, defined by $f\in\ko_{\Sigma,z}$.

(1) If  $g\in I^s(C,z)$$f$, then $f+tg$ are topologically equivalent for all but finitely many $t\in\C$ and $f+tg$ defines an equisingular deformation of $(C,z)$ over $(\C,0)$
along the trivial section.

(2) We get \index{ideal!singularity!topological}\index{singularity!ideal!topological} \index{topological!singularity ideal}\index{$is$@$I^s$}
$$I^s(C,z)=I^s(f)=\left\{g\in\ko_{\Sigma,z}\,\left|\,
\renewcommand{\arraystretch}{1}
  \begin{array}{c}
f+\varepsilon g\ \text{is a straight es-deformation of } (C,z)\\
 \text{ over}\ T_\varepsilon  \text{ (along the trivial section)}
 \end{array}\right.
\!\!\:\right\}$$
and $I^s(C,z)/I^{ea}_{fix}(C,z)$ is the tangent space to
the functor of (isomorphism classes of) straight es-deformations.

(3) We have the inclusions $\Iea_{fix}(C,z) \subset \Iea(C,z)\subset \Ies(C,z)$ and
$$\Iea_{fix}(C,z)\subset I^s(C,z)\subset \Ies_{fix}(C,z) \subset \Ies(C,z).$$\index{$ies$@$\Ies$}
$\uDef^{es}_{(C,z)}(T_\eps) = \Ies(C,z)/ \Iea(C,z)$ resp.  $\uDef^{es,fix}_{(C,z)}(T_\eps) =\Ies_{fix}(C,z)/\Iea_{fix}(C,z)$ is the tangent space
to the functor of isomorphism classes of es-deformations resp. of es-deformations with (trivial) section. The dimensions satisfy
$$ \dim_\C \Ies(C,z)/\Ies_{fix}(C,z) = \dim_\C \Iea(C,z)/ \Iea_{fix}(C,z) = 2,$$
$$ \dim_\C \Ies_{fix}(C,z)/\Iea_{fix}(C,z) = \dim_\C \Ies(C,z)/ \Iea(C,z).$$
Moreover, the forgetful morphism 
  $$\Ies_{fix}(C,z)/\Iea_{fix}(C,z) \to  \Ies(C,z)/\Iea(C,z)$$
  is an isomorphism.
\end{corollary}

\begin{proof} (1) follows from the proof of Lemma \ref{1.4}, (2) from the proof of Lemma \ref{ln1.23} applied to $f + tg, t^2 = 0$.

(3) The inclusion  
$\Iea_{fix}(C,z)\subset I^s(C,z)$ follows form the fact that for 
$g \in \Iea_{fix}(C,z)$ the deformation
$f+ \eps g$ is trivial along the trivial section, hence straight, over $T_\eps$. The other inclusions follow from the definitions.
For the first dimension statements see \cite[Lemma I.2.13]{GLS18}. The second follows from
the first and the following exact sequences,
$$0  \to   \Ies_{fix}(C,z)/\Iea_{fix}(C,z) \to \Ies(C,z)/\Iea_{fix}(C,z)
  \to  \Ies(C,z)/\Ies_{fix}(C,z)  \to 0,$$
 $$0  \to   \Iea(C,z)/\Iea_{fix}(C,z) \to \Ies(C,z)/\Iea_{fix}(C,z)
  \to  \Ies(C,z)/\Iea(C,z)  \to 0.$$ 
The statement about the forgetful morphism follows, since both spaces have the same dimension and the cokernel is $\Ies(C,z)/j(f)+\Ies_{fix}(C,z) = 0$.
\end{proof}

\medskip

\begin{remark}\label{rm:tes}
If a deformation functor has a semiuniversal object, then the tangent space to the functor (see \cite[App. C]{GLS07}) coincides with the Zariski tangent space of the semiuniversal deformation.
We set
 $$ 
\begin{array}{cllll}
\tau^s(C,z)&:=& \dim_\C\ko_{\Sigma,z}/I^s(C,z),\\
\tes(C,z)&:=&  \dim_\C\ko_{\Sigma,z}/\Ies(C,z),\\
\tes_{fix}(C,z)&:=& \dim_\C\ko_{\Sigma,z}/\Ies_{\text{fix}}(C,z),\\
\tea(C,z)&:=& \dim_\C\ko_{\Sigma,z}/\Iea(C,z),\\
\tea_{fix}(C,z)&:=& \dim_\C\ko_{\Sigma,z}/\Iea_{\text{fix}}(C,z),
\end{array}
$$
with 
 $$ 
\begin{array}{cllll}
\tea(C,z) &=& \dim_\C \uDef_{(C,z)}(T_\eps) \  \text{(usual deformations)},\\
\tea(C,z)_{fix} &=& \dim_\C \uDef^{sec}_{(C,z)}(T_\eps) \   \text{(deformations with section)},\\ 
\tea(C,z)-\tes(C,z) &=& \dim_\C \uDef^{es}_{(C,z)}(T_\eps) \  \text{(es-deformations),}\\
\tea(C,z)-\tes_{fix}(C,z)& =& \dim_\C \uDef^{es,fix}_{(C,z)}(T_\eps)\   \text{(es-deformations with section),}\\
\tea(C,z)-\tau^s(C,z) &=& \dim_\C \uDef^{s}_{(C,z)}(T_\eps) \  \text{(straight es-deformations)}.
\end{array}
$$

By Corollary \ref{cor:Is}  we have 
$$\tes(C,z)-\tes_{fix}(C,z) = \tea(C,z)-\tea_{fix}(C,z) = 2.$$ 
Moreover, by Lemma \ref{ln1.25},
$$ \tau^s(C,z) = \sum_{q\in\sT^*(C,z)}\frac{m_q(m_q+1)}{2}\ ,$$
with $m_q$ the multiplicity of the strict transform of $(C,z)$ at $q$.
\end{remark}

Let us recall the following result about straight es-deformations from \cite[Proposition II.2.69]{GLS07}, basically due to Wahl \cite{Wahl}.

\begin{proposition} \label{pr.straight}
Let \mbox{$(C,z)\subset (\Sigma,z)$} be a reduced plane curve
singularity defined by $f\in\ko_{\Sigma,z}$.
Then the following are equivalent:
\begin{enumerate}\itemsep2pt
\item[(a)] There are \mbox{$\tau'=\tau(C,z)-\tes(C,z)$} elements
  \mbox{$g_1,\dots,g_{\tau'}\in \Ies(C,z)$}
  such that
$$
\varphi^{\text{\it es}}: V\left( f+  \sum\nolimits_i t_ig_i\right)
\subset(\Sigma\!\times\C^{\tau'}\!,(z,\bo)) \xrightarrow{\phantom{a}\pr\phantom{a}}
(\C^{\tau'}\!,\bo)
$$
is a semiuniversal equisingular deformation of $(C,z)$.
\item[(b)] There exist \mbox{$g_1,\dots,g_{\tau'}\in \Ies(C,z)$} inducing a basis of
$\Ies(C,z)/\langle f,j(f)\rangle$ such that
$$
\varphi^{\text{\it es}}:V\left( f+
  \sum\nolimits_i t_ig_i\right)
\subset(\Sigma\!\times\C^{\tau'}\!,(z,\bo)) \xrightarrow{\phantom{a}\pr\phantom{a}}
(\C^{\tau'}\!,\bo)
$$
is a semiuniversal equisingular deformation of $(C,z)$.
\item[(c)] Each locally trivial deformation of the reduced exceptional divisor
  $E$ of a minimal embedded resolution of \mbox{$(C,z)\subset
  (\Sigma,z)$} is trivial.
\item[(d)] $\Ies(C,z)=\langle f,j(f), I^s(C,z) \rangle$.
\item[(e)] Each equisingular deformation of \mbox{$(C,z)$} is straight equisingular.

\end{enumerate}
\end{proposition}

Note that straight equisingular deformations  are deformations with section 
while equisingular deformations are deformations without section. Forgetting the section we get a morphism from the category $\Def^{s}_{(C,z)}$ of straight es-deformations to the category $\Def^{es}_{(C,z)}$ of es-deformations. Saying that 
an equisingular deformation is straight equisingular means that it is isomorphic 
(as deformation without section) to the image in $\Def^{es}_{(C,z)}$ of a straight equisingular deformation. 

The image $\Def^{ws}_{(C,z)}$ of $\Def^{s}_{(C,z)}$ in $\Def^{es}_{(C,z)}$
is called the category of {\em straight es-deformation without section}. 
It follows that
$$ \uDef^{ws}_{(C,z)}(T_\eps) = \langle j(f),I^s(C,z)\rangle \, / \, \langle f, j(f)\rangle$$
 is the tangent space to the functor of isomorphism classes of straight equisingular deformation of $(C,z)$ without section.
\medskip 

We mention also the following consequence for semiquasi-homogeneous and Newton nondegenerate singularities from \cite[ Proposition II.2.17 and Corollary II.2.71]{GLS07}.

\begin{corollary}\label{cor.nnd}
 Let $(C,z) \subset (\Sigma,z)$ be a reduced plane curve
singularity defined by $f\in \ko_{\Sigma,z}$, and let
\mbox{$\tau'=\tau(C,z)-\tes(C,z)$}.
\begin{enumerate}
\item[(a)] If \mbox{$f=f_0+f'$} is semiquasi-homogeneous with principal
  part $f_0$ being quasi-homogeneous of type $(w_1,w_2;d)$,  then 
  $$\Ies(C,z) =  \langle j(f), I^s(C,z)\rangle =  \langle j(f),\,  x^\alpha y^\beta \mid w_1 \alpha + w_2 \beta \geq d \rangle$$
  and a semiuniversal  equisingular deformation for $(C,z)$ is given by
$$
\varphi^{\text{\it es}}: V\left( f+  \sum\nolimits_{i=1}^{\tau'} t_ig_i\right)
\subset(\Sigma\!\times\C^{\tau'}\!,(z,\bo)) \xrightarrow{\phantom{a}\pr\phantom{a}}
(\C^{\tau'}\!,\bo)\,,
$$
for suitable \mbox{$g_1,\dots,g_{\tau'}$} representing a $\C$-basis for the quotient
$$\langle j(f),\,  x^\alpha y^\beta \mid w_1 \alpha
+ w_2 \beta \geq d \rangle\big/\ \langle f, j(f) \rangle\,.$$
\item[(b)]  If $f$ is Newton non-degenerate with Newton diagram
  $\Gamma(f)$ at the origin, then  
 \[\quad \quad \quad \ \ \Ies(C,z) =  \langle j(f), I^s(C,z)\rangle
 =  \langle j(f),\,  x^\alpha y^\beta \mid x^\alpha y^\beta
 \text{ has Newton order}\geq 1\rangle \]
and  a semiuniversal  equisingular deformation for $(C,z)$ is given by 
$$ 
\varphi^{\text{\it es}}: V\left( f+  \sum\nolimits_{i=1}^{\tau'} t_ig_i\right)
\subset(\Sigma\!\times\C^{\tau'}\!,(z,\bo)) \xrightarrow{\phantom{a}\pr\phantom{a}}
(\C^{\tau'}\!,\bo)\,,
$$
for suitable $g_1,\dots,g_{\tau'}$ representing a monomial $\C$-basis for the quotient
$$\langle j(f),\,  x^\alpha y^\beta \mid x^\alpha y^\beta
  \text{ has Newton order }\geq 1\rangle\big/\langle f, j(f)\rangle \,.$$
\end{enumerate}
Moreover, in both cases each equisingular deformation of \mbox{$(C,0)$} is
straight equisingular.

\end{corollary}

\begin{remark}\label{rm:lin}
(1) We can extend  the chosen basis of $\Ies(C,z)/\langle f,j(f)\rangle $ to a
basis $g_1,\dots,g_{\tau}$ of $\ko_{\Sigma,z}/\langle f,j(f)\rangle $, showing that
the base space of the semiuniversal es-deformation is the
linear subspace $\{t_{\tau'+1}=\cdots =t_{\tau} =0\}$ of the usual semiuniversal deformation of $(C,z)$ given by  $f+  \sum\nolimits_{i=1}^{\tau} t_ig_i$.

(2) Note that in  Proposition \ref{pr.straight}(b) and Corollary \ref{cor.nnd} not every (monomial) basis of $\Ies(C,z)/\langle f,j(f)\rangle $ has the claimed property (this was not clearly formulated in \cite{GLS07}). We illustrate this by an example (due to Marco Mendes), using {\sc Singular} 
\cite{DGPS}.
\end{remark}

\begin{example}
Let $f=(y^3+x^7)(y^3+x^{10})$, which is Newton-nondegenerate. The quotient
$ Q := \Ies(C,z)/\langle f,j(f)\rangle $ is the tangent space to the stratum of es-deformations  ($\mu$-constant stratum) of $f$, but in general not the stratum itself. We show that the monomial $x^{10}y^2$ is part of a monomial basis of $Q$.  It is, however, below the Newton diagram (the monomials on or above the Newton diagram are those of Newton order\footnote{We say that a monomial has Newton order  $d \in \R$  (w.r.t. $f)$ iff it corresponds to a point on the hypersurface  
 $d \cdot \Gamma(f) \subset \R^2$.}  $\geq 1$)  and we will see that  the deformation $f + tx^{10}y^2$ is not equisingular. At the end we construct a basis of monomials of $Q$ of Newton order $\geq 1$, explaining some of the peculiarities of {\sc Singular}.
\medskip

\noindent
\begin{verbatim}
LIB "all.lib";                //loads all libraries
ring r  = 0,(x,y),ds;         //local ring with degree ordering
poly f = (y3+x7)*(y3+x10);
list L = esIdeal(f,1);        //computes all 3 es-ideals:
                              //L[1]: = I^{es} = es-ideal of Wahl                                            
                              //L[2]: = I^{es}_fix                                                   
                              //L[3]: = I^s = top. singularity ideal 
ideal Ifix = std(L[2]);       //we continue with I^{es}_{fix}
Ifix;
//-> Ifix[1]=2xy5+x8y2+x11y2
//-> Ifix[2]=y6
//-> Ifix[3]=x5y4
//-> Ifix[4]=x7y3
//-> Ifix[5]=x10y2
//-> Ifix[6]=x14y
//-> Ifix[7]=x17
ideal tj = std(jacob(f)+f);   //the Tjurina ideal
ideal Ies = std(L[1]);        //this is I^{es}
NF(Ies,tj)+0;                 //the non-zero elements not in tj
//-> _[1]=x5y4
//-> _[2]=x10y2
//-> _[3]=x14y
\end{verbatim}
This shows that the monomial $x^{10}y^2$ is part of a monomial basis of $Q$. The deformation $f + tx^{10}y^2$  is not equisingular as we check by computing the Milnor numbers:

\noindent
\begin{verbatim}
milnor(f);
//-> 71
ring rt  = (0,t),(x,y),ds;       
poly ft = (y3+x7)*(y3+x10) + t*x10y2;
milnor(ft);
//-> 70
\end{verbatim}

\noindent To compute a monomial bases of $Q$ with Newton order $\geq 1$, we first construct a set of monomials containing the desired basis:
\noindent
\begin{verbatim}
setring r;    
int d = deg(highcorner(tj));  //determines the monomials inside tj
ideal k; int i;
for (i = 1; i <= d; i++)
{   k=k+maxideal(i);
}
\end{verbatim}
The ideal $k$ contains all monomials  of degree $\leq d$. The monomials of degree $ > d$  belong to  the ideal $tj$ (cf. \cite[Lemma 1.7.17]{GP07}).
\medskip

\noindent
To compute the monomials in $k$ on or above the Newton diagram, let $n1, n2$ be the affine linear polynomials defining the faces of the Newton diagram. 
We compute the Newton order of the monomials by computing their exponents and then take the minimum value of $n1$ and $n2$ at these exponents.
\begin{verbatim}
poly n1 = y + 3/10x - 51/10; 
poly n2 = y + 3/7x - 6;      
int e1,e2;
poly mi ;
ideal J;
for (i = 1; i<=size(k); i++)
{ e1 = leadexp(k[i])[1];
  e2 = leadexp(k[i])[2];
  mi = min(subst(n1,x,e1,y,e2),subst(n2,x,e1,y,e2));
  if (mi >=0)
  {J = J + k[i];}
}
\end{verbatim}
The elements of $J$ generate $Q$, all have  Newton order $\geq 1$. We can now of course get a basis out of $Q$ by homogenization and applying linear algebra methods.
The {\sc Singular} command  $\mathtt {reduce(J,tj);}$ creates a basis of $Q$, however, the reduction process via standard bases produces $x^{10}y^2$, which is a monomial of the generator $2y^5+x^7y^2+x^{10}y^2$ of $tj$, as part of this basis.  A good chance to get the right basis with  elements of Newton order $\geq 1$ is to use the reduction process in a ring with a $dp$ ordering, since this ordering prefers higher degrees (and works in our example).
\begin{verbatim}
ring R = 0, (x,y),dp;
ideal J = imap(r,J);
ideal tj = std(imap(r,tj));
ideal N = reduce(J, tj);
N = normalize(N+0);   N;        //just to get a nice-looking basis
//->  N[1]=x3y5
//->  N[2]=x4y5
//->  N[3]=x5y4
//->  N[4]=x5y5
//->  N[5]=x6y4
//->  N[6]=x14y
//->  N[7]=x15y
\end{verbatim}
The 7 elements of $N$ have all Newton order $\geq 1$. We check the dimension by computing $\dim_\C Q = \tau(f) - \tes(f) $.
\begin{verbatim}
setring r;
int tau = tjurina(f); tau;
//-> 59
int tes = vdim (Ies); tes;
//-> 52
tau - tes;
// -> 7
\end{verbatim}
\end{example}

We show now that the functor of straight es-deformations is in general a linear subfunctor of all es-deformations (with and without section).
\medskip

\begin{proposition} \label{pr.straight.su}
Let \mbox{$(C,z)\subset (\Sigma,z)$} be a reduced plane curve
singularity defined by $f\in\ko_{\Sigma,z}$.

(1) There exist \mbox{$g_1,\dots,g_{\tau'}\in I^s(C,z)$} inducing a basis of
$I^s(C,z)/\langle f, \fm_z j(f)\rangle$ such that
$$
\varphi^{\text{\it s}}:V\left( f+
  \sum\nolimits_i t_ig_i\right)
\subset(\Sigma\!\times\C^{\tau'}\!,(z,\bo)) \xrightarrow{\phantom{a}\pr\phantom{a}}
(\C^{\tau'}\!,\bo)\,,
$$
is a semiuniversal straight equisingular deformation of $(C,z)$. 

(2) There exist $g_1,\dots,g_{\tau''}\in \langle j(f), I^s(C,z)\rangle$ 
inducing a basis of of the quotient
$ \langle j(f),I^s(C,z)\rangle \, / \, \langle f, j(f)\rangle$ such that
$$
\varphi^{\text{\it s}}:V\left( f+
  \sum\nolimits_i t_ig_i\right)
\subset(\Sigma\!\times\C^{\tau''}\!,(z,\bo)) \xrightarrow{\phantom{a}\pr\phantom{a}}
(\C^{\tau'}\!,\bo)\,,
$$
is a semiuniversal straight equisingular deformation of $(C,z)$ without section.

\end{proposition}

\begin{proof}
(1) That the deformation $\varphi^s$ is straight es follows from Lemma \ref{ln1.23} and its proof.  Extending  the chosen basis of $I^s(C,z)/\langle f, \fm_zj(f)\rangle$ to a
basis of $\ko_{\Sigma,z}/\langle f, \fm_zj(f)\rangle$
shows that the base space of $\varphi^s$
 is a linear subspace $B^s_{C,z}$ of the base space $B^{fix}_{C,z}$ of the  semiuniversal deformation of $(C,z)$ with (trivial) section.
If a given deformation over an arbitrary base germ $(T,0)$ is straight es, it can be induced by a map $\varphi : (T,0) \to B^{fix}_{C,z}$. $\varphi$ \  must factor through $B^s_{C,z}$
by definition of $I^s$ because otherwise the deformation is not equimultiple along the trivial sections through all infinitely near points $q \in \kt^*(C,z)$. This shows that
$\varphi^{\text{\it s}}$ is complete.
Moreover, for extensions of Artinian base spaces  $(T,0) \subset (T',0)$, the equimultiple deformations along the trivial sections over  $(T,0)$ can
can be clearly extended to equimultiple deformations along the trivial sections over
$(T',0)$. Hence the versality follows.
(For ``complete'' and ``versal'' see \cite [Definition II.1.8]{GLS07}). 

(2) follows from (1), noting that  $ \langle j(f),I^s(C,z)\rangle \, / \, \langle f, j(f)\rangle$ is the tangent space to the functor of isomorphism classes of straight equisingular deformation of $(C,z)$ without section. As in Remark  \ref{rm:lin}
the base space of the semiuniversal straight equisingular deformation of $(C,z)$ without section can be realized as a linear subspace of the usual semiuniversal   deformation of $(C,z)$.
\end{proof}

\section {Relation to the Rooted Hilbert Scheme}
In this section we show that straight es-deformations of a plane curve singularity 
$(C,z) \subset (\Sigma,z)$ appear as fibres of a morphism from all es-deformations to the ``rooted Hilbert scheme'', a certain punctual Hilbert scheme in $\Sigma$ parametrizing topological singularity schemes of plane curve singularities with fixed cluster graph. 
Recall that the Hilbert scheme of points on $\Sigma$ parametrizes families of zero-dimensional schemes,
 i.e., complex subspaces of $\Sigma$ concentrated on finitely many points.

\begin{definition}
Let $T$ be an arbitrary complex space. A {\em family
of zero-dimensional schemes (of degree $n$) on $\Sigma$
over $T$\/} is a commutative diagram of complex spaces
$$
\SelectTips{cm}{12}
\xymatrix@C=6pt@R=10pt{ \kz \:\:\ar@{^{(}->}[rr]^-{j}
 \ar[rd]_\varphi
 & & \:\Sigma \times T \ar[dl]^{\pr}\\
&  T }
$$
with  $j : \kz \hookrightarrow \Sigma \times T$ a closed embedding, $\pr$
the projection and $\varphi = \pr \circ j$ finite and flat.
The fibre
$\kz_t:=\varphi^{-1}(t)$\index{$zzt$@$\sZ_t$}, $t \in T$, is mapped by $j$ to a
zero-dimensional subscheme of $\Sigma$ and if its degree is $n$ we say that
$j : \kz \hookrightarrow \Sigma \times T$ is a {\em family of zero-dimensional schemes of degree n}.
\end{definition}

We usually identify two families $j : \kz \hookrightarrow \Sigma \times T$  and $j' : \kz ' \hookrightarrow \Sigma \times T$ if the images $j( \kz)$ and $j'(\kz ')$ coincide, and then we get subspaces $\kz \subset\Sigma \times T$ (of degree $n$) which are finite and flat over $T$. For $\Sigma$ algebraic and algebraic families this is the Hilbert functor:

The {\em Hilbert functor } $\Hilb^n_{\Sigma}$ associates
to an algebraic variety $T$ the set
\begin{eqnarray*}
 \Hilb^n_{\Sigma}(T) &:=& \biggl\{
\begin{array}{c}
 \mbox{closed algebraic subvarieties } \kz \subset\Sigma \times T \\
  \mbox{which are finite of degree } n \mbox{ and flat over } T
\end{array}
\biggr\}.
 \end{eqnarray*}
 
Hence $ \kz \subset\Sigma \times T \to T$ is a family of zero-dimensional schemes of degree n.
If $\Sigma$ is a projective surface, the Hilbert functor is known to be representable by a  smooth projective algebraic variety of dimension $2n$, denoted by \mbox{$\tHilb^n_\Sigma$}\index{$hilbnSigma$@$\tHilb^n_\Sigma$}, the
{\em Hilbert scheme\/}\index{Hilbert!scheme!of points} {\em of points in} $\Sigma$.
Representability means that there exists a closed algebraic subvariety \mbox{$\sU^n_{\Sigma} \subset  \:\Sigma \times \tHilb^n_{\Sigma}$}, finite of degree $n$ and flat over $\tHilb^n_{\Sigma}$, such that each element of  \mbox{$\Hilb^n_{\Sigma} (T)$} can be induced from the universal family \mbox{$\sU^n_{\Sigma} \to \tHilb^n_{\Sigma}$}
via base change by a unique algebraic morphism \mbox{$T \to  \tHilb^n_\Sigma$}.
\medskip

There exists a birational
morphism, the {\em Hilbert-Chow morphism}\index{Hilbert-Chow morphism}
$$ \Phi: \tHilb^n_\Sigma \longrightarrow \text{Sym}^n \Sigma \,,
$$ 
$\text{Sym}^n \Sigma$ denoting the $n$-th symmetric power of $\Sigma$.
It assigns to a closed subscheme \mbox{$Z\subset
  \Sigma$} of degree $n$ the $0$-cycle consisting of the points of $Z$
with multiplicities given by the length of their local rings on $Z$. If $Z$ consists of $n$ distinct simple
points $z_1,...,z_n$ then $\Phi$ maps $Z$ to the image of $(z_1,...,z_n)
\in\Sigma^n$ in $\Sym^n\Sigma=\Sigma^n/S_n$ where $S_n$ is the symmetric group permuting the coordinates. 
\medskip

\begin{definition}{\bf (Punctual Hilbert functor and scheme)}
\label{defPunctHilb}
Fix \mbox{$z \in \Sigma$} and let $x,y$ be local coordinates at $z$.
Then we have the {\em punctual Hilbert functor}
$\Hilb^n_{\C\{x,y\}}$,
which associates to a complex space $T$ the set
$$ \Hilb^n_{\C\{x,y\}}(T) := \bigl\{ \!\; (\kz \subset
\Sigma\times T) \in \Hilb^n_{\Sigma}(T) \:\big|\:
\text{supp}(\kz) \subset \{z\}\times T \!\;\bigr\}.$$
\end{definition}

Fogarty \cite{Fogarty} and Hartshorne \cite{Hartshorne} showed that the
punctual Hilbert functor is representable (in the algebraic category)
by a connected projective variety
\mbox{$\tHilb^n_{\C\{x,y\}}$},
the {\em punctual Hilbert scheme\/},
which can be identified with the closed subvariety
\mbox{$\Phi^{-1}(n\cdot z) \subset \tHilb^n_\Sigma$}.
Brian\c{c}on \cite{Briancon} proved (in the analytic categroy) that
$\tHilb^n_{\C\{x,y\}}$ is irreducible
(in general not reduced) of dimension $n-1$. 

\begin{remark} \label{douady}
A functor analogous to the Hilbert functor can be defined for complex spaces $T$, associating to $T$ the set of closed complex subspaces $\kz \subset\Sigma \times T$, finite of degree $n$ and flat over $T$. This is the {\em Douady functor}, denoted by 
$\Dou^n_{\Sigma}$, which is representable by a complex space $\tDou^n_{\Sigma}$, the {\em Douady space}. It has a universal family \mbox{$\sV^n_{\Sigma} \to \tDou^n_{\Sigma}$} with the same universal property as the Hilbert functor but for morphisms of complex spaces, i.e.  each element of  \mbox{$\Dou^n_{\Sigma} (T)$} can be induced from \mbox{$\sV^n_{\Sigma} \to \tDou^n_{\Sigma}$} by a unique 
analytic morphism \mbox{$T \to  \tDou^n_\Sigma$}.

The Hilbert scheme and the Douady space exist in much greater generality and their relation is discussed in detail in \cite[Section II.2.2.1]{GLS18}. We just mention that  for families of zero-dimensional schemes on the smooth projective 
surface $\Sigma$, the analytification of
$\tHilb^n_{\Sigma}$ and $\tDou^n_{\Sigma}$ are isomorphic as complex spaces 
and that $\sU^n_{\Sigma} \subset  \:\Sigma \times \tHilb^n_{\Sigma}$ has 
the universal property for complex spaces $T$ and analytic morphisms $T \to \tHilb^n_{\Sigma}$
(proved in greater generality in \cite[Proposition II.2.28]{GLS18}).

We will consider the punctual Hilbert functor as above with $z$ a point in an
arbitrary smooth complex analytic surface $\Sigma$. Since we consider families 
 \mbox{$(\kz \subset \Sigma\times T)$} $\in \Hilb^n_{\Sigma}(T)$ with  
 $\text{supp}(\kz) \subset \{z\}\times T$, we may assume that $\Sigma$ is projective and
the above mentioned results hold in the algebraic as well as in the analytic category and we will always write $\tHilb$, also in the analytic situation.
\end{remark}

In order to define the ``rooted Hilbert functor'', a subfunctor of the punctual Hilbert functor,  we have to extend the definition of a cluster graph to arbitrary zero-dimensional subschemes of $Z \subset \Sigma$. We just sketch the definition and refer to \cite{GLS18} for details.

We set $\deg Z =  \dim_\C H^0(Z,\ko_Z) $, the degree of $Z$,  
and $\mt(Z,z)$ the minimum order at $z$ of the elements contained in the ideal $I_Z$ defining $Z$, the multiplicity of $Z$. 
Given a zero-dimensional scheme
\mbox{$Z_{0}:=Z\subset \Sigma=:\Sigma_{0}$}, let
\mbox{$\pi_{i}:\Sigma_{i}\to \Sigma_{i-1}$} be the blowing up of
\mbox{$\text{supp}(Z_{i-1}) \subset \Sigma_{i-1}$}, and let $Z_{i}$ be the
strict transform
of \mbox{$Z_{i-1}$},
\mbox{$i=1,\dots,r$}. Note that \mbox{$\text{supp}(Z_{i})$} consists of
finitely many infinitely near points (of level $i$) and that
  \mbox{$\deg Z_{i}$} is strictly decreasing (which can be seen, for
  instance, by using standard bases).
 We choose
$r$ minimal with the property \mbox{$\text{supp}(Z_{r})=\emptyset$}.
Then the (isomorphism class of the) sequence 
$$ \sK:=\bigl(\text{supp}(Z_{0}),\pi_1,\text{supp}(Z_{1}),\dots,\pi_{r-1},
\text{supp}(Z_{r-1})\bigr) $$ 
is called a constellation on $\Sigma$ (generalizing constellations for plane curve singularities defined in Section 1).
Setting 
$$m_q:=\mt(Z_{i},q)\quad\text{for each}\quad q\in \text{supp}(Z_{i})\ ,$$ 
and $\bm = (m_q)_q$ then  \mbox{$\Cl(Z):= (\sK,{\bm})$} is
called the {\em cluster defined by the zero-di\-men\-sional
  scheme} $Z$. Similar as in Definition \ref{def:tree} we define the
  graph $\Gamma_\sK$ of $\sK$ and a proximity relation
  ($q_j\dashrightarrow q_i$) on the points of $\sK$ and call the triple 
$$\clg(Z):=\clg(\Cl(Z)) := (\Gamma_{\sK},\dashrightarrow, {\bm})$$
the {\em cluster graph of the zero-dimensional scheme $Z$.}  
We set $n\,:=\: \sum_{q \in \sK} \frac{m_q(m_q \!\!\:+\!\!\; 1)}{2}\,.$
The support of $Z$ is called the {\em root of } $\G$.\\

Let $\!\;\G$ be the cluster graph of a zero-dimensional
scheme on $\Sigma$.
We define the {\em Hilbert functor with fixed cluster
graph $\!\;\G$} on the category of reduced complex spaces as 
$$ \Hilb^{\G}_{\Sigma}(T)= 
\bigl\{ (\kz \subset  \Sigma\times T)\in \Hilb^n_{\Sigma}(T)
  \,\big|\:\clg(\kz_t)=\G \text{ for all } t\in T \bigr\}.$$

It is proved in \cite[Proposition I.1.52]{GLS18} that for 
$(\kz \subset  \Sigma\times T)\in \Hilb^{\G}_{\Sigma}(T)$
there exists a complex space $T'$  and a finite surjective morphism 
$\alpha:T'\to T$
such that the induced family $\alpha^* \kz \to T'$ is resolvable by 
a sequence of blowing ups 
$\pi_i:\sX^{(i+1)} \!\to \sX^{(i)}$
along equimultiple sections $\sigma^{(i)}_j:T'\to \sX^{(i)}$  ($i = 0, \cdots, N, j = 0, \cdots k_i$) with $\sX^{(0)}=\Sigma\times T, \kz^{(0)}\!=\kz$  and
$\kz^{(i+1)}$ the strict transform of $\kz^{(i)}$.
Moreover, $\text{supp}( \kz^{(i)}) = \bigcup_{j=1}^{k_i}\sigma^{(i)}_j(T)$ and
$\text{supp}(\kz^{(N+1)}) =\emptyset$.
The initial sections $\sigma^{(0)}_j$ satisfy 
$\{ \sigma^{(0)}_1(t),\cdots,\sigma^{(0)}_{k_1}(t)\}=$ root$(\G(\kz_t))$.
The rooted Hilbert functor is the subfunctor of $ \Hilb^{\G}_{\Sigma}$ where the
initial sections are trivial.\\

Consider now a reduced plane curve singularity $(C,z) \subset (\Sigma,z)$ with
essential tree $\sT^*$ and  cluster graph
$\G(C,z) = (\Gamma_{\sT^\ast},\dashrightarrow, {\bm})$, defined in Section 1. 
The definition of the 
cluster graph of a zero-dimensional scheme is modeled in such a way that we have
$$\G(C,z) = \clg(Z^s(C,z)),$$
where $Z^s(C,z)$ is the topological singularity scheme from 
Definition \ref{def:topsing}. The root of $\G(C,z)$ is \{z\}.

\begin{definition}{\bf (Rooted Hilbert functor)} \label{def:hilbroot}
Let $\G = \G(C,z)$ be the cluster graph of a reduced plane curve singularity. 
The {\em rooted Hilbert functor} or the {\em punctual Hilbert functor fixing $\G$}
is the subfunctor $ \Hilb^{\G}_{\C\{x,y\}} $ of  $ \Hilb^{\G}_{\Sigma}$ associating to
a reduced complex space $T$ the subspaces  
$(\kz \subset  \Sigma\times T) \in \Hilb^\G_{\Sigma}(T)$ 
such that the initial section $\sigma:=\sigma^{(0)}_1$ passing through $z$ is trivial.
\end{definition}

The main result about the rooted Hilbert functor is the following theorem. It is used in \cite[Section IV.6]{GLS18} to prove asymptotically sufficient conditions for the irreducibility of the variety of plane projective curves of given degree with a fixed number of singularities of given topological type.

\begin{theorem} {\bf (\cite[Theorem I.1.55, I.1.57]{GLS18})}\label{thm:h0}
The rooted functor $\Hilb^{\G}_{\C\{x,y\}}$ is
representable by  $\tHilb^{\G}_{\C\{x,y\}}$, the {\normalfont rooted Hilbert scheme}, an irreducible, quasi-projective subvariety
of the projective variety $\tHilb^n_{\C\{x,y\}}$,
of dimension equal to the number of free vertices
in $\sT^*(C,z) \setminus\{z\}$.
\end{theorem}

Let $S$ denote the topological type of $(C,z)$. Since
$\G =\G(C,z)$ is a complete invariant of $S$, $ \Hilb^{\G}_{\C\{x,y\}} $ 
depends only on $S$ and
we can introduce the {\em punctual Hilbert scheme associated to a 
topological type S},
$$ \kh_0^s(S) := \tHilb^{\G}_{\C\{x,y\}}\,. $$

It is shown in \cite[Proposition I.1.17]{GLS18} that each point in $ \kh_0^s(S)$ corresponds to a topological singularity 
scheme $Z^s(C,w)$  of a plane curve singularity $(C,w)$ 
of topological type $S$ and with
$\deg Z^s(C,w) = \dim_\C \ko_{\Sigma,z}/ I^s(C,w) = n$, with 
$n = \sum_{q\in \sT^*(C,w)} \frac{m_q(m_q\!\!\:+\!\!\;1)}{2}$,
$m_q$ the multiplicity of the strict transform of $(C,w)$ at $q$,
and with fixed cluster graph 
$\G =\clg(Z^s(C,w))$.\\

Before we relate $\kh^s_0(S)$ to straight equisingular deformations, we consider
several related semiuniversal base spaces.
\medskip

\begin{lemma}\label{lem:bases}
For a reduced plane curve singularity $(C,z)$ the following base spaces of semiuniversal deformations of $(C,z)$ are smooth of the given dimension.
\begin{itemize}
\item [(1)] $B_{C,z}$  the base space of the (usual) semiuniversal deformation\\
of dimension $\tea(C,z) = \tau(C,z)$, \index{$tau$@$\tau(C,z)$}
\item  [(2)] $B^{fix}_{C,z}$  the base space of the semiuniversal deformation with section\\
of dimension $\tea_{fix}(C,z) = \tau_{fix}(C,z)$,\index{$taueaf$@$\teaf$}
\item  [(3)] $B^{es}_{C,z}$  the base space of the seminuniversal equisingular deformation\\
of dimension $\tea(C,z)-\tes(C,z)$,\index{$taues$@$\tes$}
\item  [(4)] $B^{es,fix}_{C,z}$  the base space of the seminuniversal es-deformation with section\\
of dimension $\tea_{fix}(C,z) - \tes_{fix}(C,z) = \tea(C,z) - \tes(C,z)$,\index{$tauesf$@$\tesf$}
\item  [(5)] $B^{s}_{C,z}$  the base space of the seminuniversal straight es-deformation\\
of dimension $\tea_{fix}(C,z)-\tau^s(C,z)$. \\
\end{itemize}
\end{lemma}

Note that the definitions of $\tea$ and $\tes$ are consistent in the following sense.
$\tes(C,z)$ is the codimension in $B_{C,z}$ of the base space of the semiuniversal
{\em equisingular} deformation of $(C,z)$ (which coincides with the $\mu$-constant stratum in $B_{C,z}$), while $\tea(C,z)$ is the codimension in $B_{C,z}$ of the base space of the semiuniversal {\em equianalytic} deformation of $(C,z)$ (which is the reduced point).

\begin{proof}
For $B_{C,z}$ and $B^{fix}_{C,z}$ the smoothness and the dimension statements  are well known.
By \cite [Theorem II.2.61, Corollary II.2.67 and Proposition II.2.63]{GLS07}
we know that $B^{es}_{C,z}$
is isomorphic to the $\mu$-constant stratum in $B_{C,z}$ and smooth, while $B^{es,fix}_{C,z}$
is isomorphic to the $\mu$-constant stratum in $B^{fix}_{C,z}$. The smoothness of $B^{s}_{C,z}$ follows from 
Proposition \ref{pr.straight.su}.
The dimension statements are a consequence of the smoothness and of 
Corollary \ref{cor:Is}, since then 
the dimension coincides with the dimension of the Zariski tangent space.

To see the smoothness of $B^{es,fix}_{C,z}$ and the second equality in (4) 
note that in general deformations with section
are isomorphic to deformations
with trivial section (cf.  \cite[Proposition II.2.2]{GLS07})
and we will identify the corresponding
semiuniversal base spaces $B^{sec}_{(C,z)} = B^{fix}_{(C,z)}$.
Moreover, the forgetful morphism of functors from
deformations with section to deformations without section is smooth 
by \cite[Corollary II.1.6]{GLS07}.
It follows that the induced morphism of base spaces
$B^{es,fix}_{C,z} \to B^{es}_{C,z}$ is smooth (flat with smooth fibre).
Hence $B^{es,fix}_{C,z}$ is smooth and of the same dimension as its tangent space,  namely
$\dim_\C T^{1,es,fix}_{C,z} = \tau_{fix}(C,z) -\tes_{fix}(C,z) =  \tau(C,z) -\tes(C,z)$
by  Corollary \ref{cor:Is}. As this is the
dimension of $B^{es}_{C,z}$ by Proposition \ref{pr.straight.su}, we get that
the morphism $B^{fix,es}_{C,z} \to B^{es}_{C,z}$ is an isomorphism
(reflecting the fact that the equisingular section is unique).
\end{proof}

We state now the main result about the relation between equisingular deformations  of $(C,z)$ along the trivial section, the induced deformation of the topological singularity scheme $Z^s(C,z)$, and straight equisingular deformations of $(C,z)$.

\begin{theorem}\label{th:h0}
With the above notations there exists a surjective morphism of germs
               $$ \psi: B^{es,fix}_{C,z} \to \kh_0^s(S), $$
with  smooth fibre $\psi^{-1}(Z^s(C,z)) \cong B^{s}_{C,z}$.
Moreover, we have
 $$ 
\begin{array}{lll}
\dim \kh_0^s(S) &=& \# \{\text{free vertices } q\in\sT^*(C,z)\setminus\{z\} \}\\
&=& \dim_\C\Ies_{\text{fix}}(C,z)/I^s(C,z)\\
&=& \deg Z^s(C,z) - \tes((C,z)-2 \\
&=& \tau^s(C,z) - \tes_{fix}(C,z)\\
&=& \sum_{q\in \sT^*(C,z)} \frac{m_q(m_q\!\!\:+\!\!\;1)}{2}\, - \tes_{fix}(C,z)
\end{array}
$$
with $m_q$ the multiplicity of the strict transform of $(C,z)$ at $q$.
\end{theorem}

\begin{proof}
Equisingular deformations $\sC \subset \Sigma \times T \to T$ of $(C,z)$ along the trivial section over a reduced germ $T$ are exactly those,  which are equimultiple along the trivial section $\sigma(t) = z$ and equimultiple along the 
(not necessarily trivial) sections 
$\sigma_{(q)},  q \in \sT^*(C,z)$, through the non-nodes of the reduced total transform of $(C,z)$ in a resolution of $(C,z)$. Equisingularity implies that 
the cluster graph $\G(\sC_t,z) =\G(Z^s(\sC_t,z))$ is constant and equal to 
$\G =\clg(Z^s(C,z))$. Hence $\deg Z^s(\sC_t,z)$ is constant and $\kz = \{Z^s(\sC_t,z)\}_{t \in T}$ is a flat family of zero-dimensional schemes in $\Hilb^{\G}_{\C\{x,y\}}(T)$,  a deformation of  $Z^s(C,z)$.  

Since the association $(\sC  \to T) \to (\kz \to T)$ is functorial in $T$ and respects isomorphism classes, we get a morphism of functors
$$ \uDef^{es,fix}_{C,z} \to \Hilb^{\G}_{\C\{x,y\}},$$
 where $\uDef^{es,fix}_{C,z}$ is the functor of isomorphism classes
of equisingular deformations of $(C,z)$ with trivial section.
A family in $\Hilb^{\G}_{\C\{x,y\}}$ is trivial if it is given by the
subspace $Z^s(C,z)\times T \subset \Sigma\times T$, i.e.,
iff the sections through the infinitely near points $q \in \sT^*(C,z)$ are trivial. 
Hence the fibres of the morphism of functors are exactly the straight equisingular deformations of $(C,z)$.

The morphism of functors induces a morphism of germs
              $ \psi: B^{es,fix}_{C,z} \to \kh_0^s(S), $
which is surjective (by \cite[Proposition I.1.17]{GLS18}) with fibre over $Z^s(C,z)$ being isomorphic to $B^{s}_{C,z}$, since every straight es-deformation of $(C,z)$
over $T$ can be induced via a map $T \to B^{es,fix}_{C,z}$ which must factor through $\psi^{-1}(Z^s(C,z))$.

By Theorem \ref{rm:tes} we get the first formula for the dimension of  $\kh_0^s(S)$.
For the others we use the theorems of Frisch and Sard to see that there are
points $Z^s(C',z')$ in $\kh_0^s(S)$ arbitrary close to $Z^s(C,z)$, 
over which $\psi$ is flat and hence satisfies
$$
\begin{array}{lll}
 \dim (\kh_0^s(S), Z^s(C',z')) &= &\tau(C',z') - \tes(C',z') - \dim_\C I^s(C',z')/\Iea(C',z')\\
 &=& \deg Z^s(C',z') - \tes_{fix}(C',z').
 \end{array}
 $$
 Since $\deg Z^s(C',z') = \deg Z^s(C,z)$ and since $\kh_0^s(S)$ is irreducible
 (by Theorem \ref{thm:h0}) its dimension is constant and thus we get $\tes_{fix}(C',z') = \tes_{fix}(C,z)$. Using Remark \ref{rm:tes} we get the last four dimension formulas. 
\end{proof}

\begin{remark} For any es-deformations there exists (unique) sections
$\sigma_{(q)}$ through $q\in\sT^*=\sT^*(C,z)$ along which the deformation of the reduced total transform is equimultiple. For straight es-deformations $f + tg$ over $(\C,0)$ with
$g\in I^s(f)$ these sections are all trivial. For arbitrary es-deformations
with tangent directions in $\Ies_{\text{fix}}(f)$
the sections through satellite points $q\in\sT^*$ have to stay at
satellite points. But the sections through free points
in $\sT^*\setminus\{z\}$ may move along the exceptional divisor, giving one degree of freedom for every free point. Since $\kh^s_0(S)$ represents es-deformations with trivial initial section, we get a geometric interpretation of the formula $\dim\kh^s_0(S)$ = $\#$(free vertices $q\in\sT^*\setminus\{z\}$).
\end{remark}

\bibliographystyle{amsalpha}

\noindent
Email: greuel@mathematik.uni-kl.de
\end{document}